\documentclass[a4paper,12pt]{amsart}

 \usepackage{hyperref}

\usepackage{microtype}
\usepackage{amssymb,amsmath,amsfonts,amsthm,bm,latexsym,amscd,verbatim,url,nicefrac,stmaryrd,enumerate,colonequals,dsfont,appendix}
\usepackage[all]{xy}
\usepackage{times}
\usepackage{graphicx}
\usepackage{fullpage}

{\bf}{\it}
\newtheorem*{thm*}{Theorem}
\newtheorem{thm}{Theorem}[section]{\bf}{\it}
\newtheorem{prop}[thm]{Proposition}
\newtheorem{lemma}[thm]{Lemma}
\newtheorem{cor}[thm]{Corollary}
\newtheorem{conj}[thm]{Conjecture}
\theoremstyle{definition}
\newtheorem{dfn}[thm]{Definition}

\theoremstyle{remark}
\newtheorem{rmk}[thm]{Remark}
\theoremstyle{remark}
\newtheorem{exm}[thm]{Example}
\newtheorem{assu}[thm]{Assumption}

\newcommand{\A}{\mathbb{A}}
\newcommand{\B}{\mathbb{B}}

\newcommand{\F}{\mathbb{F}}

\newcommand{\HH}{\mathbb{H}}
\newcommand{\LL}{\mathbb{L}}
\newcommand{\N}{\mathbb{N}}
\newcommand{\Q}{\mathbb{Q}}
\newcommand{\RR}{\mathbb{R}}
\newcommand{\R}{\mathbb{R}}
\newcommand{\T}{\mathbb{T}}
\newcommand{\Z}{\mathbb{Z}}

\newcommand{\cat}{\mathbf{C}}
\newcommand{\catD}{\mathbf{D}}

\newcommand{\ra}{\rightarrow}
\newcommand{\mcB}{\mathcal{B}}

\newcommand{\mcF}{\mathcal{F}}

\newcommand{\mcO}{\mathcal{O}}
\newcommand{\mcU}{\mathcal{U}}
\newcommand{\mcV}{\mathcal{V}}

\newcommand{\adj}[4]{#1\negmedspace: #2\rightleftarrows #3:\negmedspace #4}

\DeclareMathOperator{\Aff}{Aff}

\DeclareMathOperator{\Aut}{Aut}

\DeclareMathOperator{\car}{char}

\DeclareMathOperator{\colim}{colim}
\DeclareMathOperator{\ct}{ct}
\DeclareMathOperator{\Cor}{Cor}

\DeclareMathOperator{\eff}{eff}
\DeclareMathOperator{\et}{\acute{e}t}

\DeclareMathOperator{\Et}{Et}

\DeclareMathOperator{\fh}{fh}

\DeclareMathOperator{\fhet}{fh\acute{e}t}

\DeclareMathOperator{\Frob}{Frob}
\DeclareMathOperator{\Frobet}{Frob\acute{e}t}

\DeclareMathOperator{\Hom}{Hom}

\DeclareMathOperator{\id}{id}

\DeclareMathOperator{\Nis}{Nis}
\DeclareMathOperator{\Nor}{Nor}

\DeclareMathOperator{\Perf}{Perf}

\DeclareMathOperator{\qfh}{qfh}
\DeclareMathOperator{\quiet}{quiet}
\DeclareMathOperator{\red}{red}

\DeclareMathOperator{\Rig}{Rig}
\DeclareMathOperator{\RigCor}{RigCor}
\DeclareMathOperator{\RigNor}{RigNor}
\DeclareMathOperator{\RigSm}{RigSm}

\DeclareMathOperator{\Sm}{Sm}
\DeclareMathOperator{\Sp}{Sp}
\DeclareMathOperator{\Spa}{Spa}
\DeclareMathOperator{\Spec}{Spec}

\DeclareMathOperator{\tr}{tr}

\DeclareMathOperator{\Ch}{\bf{Ch}}

\DeclareMathOperator{\DA}{\bf{DA}}
\DeclareMathOperator{\DM}{\bf{DM}}

\DeclareMathOperator{\RigDA}{\bf{RigDA}}
\DeclareMathOperator{\RigDM}{\bf{RigDM}}

\DeclareMathOperator{\Psh}{\bf{Psh}}
\DeclareMathOperator{\PST}{\bf{PST}}

\DeclareMathOperator{\Sh}{\bf{Sh}}

\DeclareMathOperator{\Ho}{Ho}

\begin{document}
\title{Effective motives with and without transfers in characteristic $p$}
\author{Alberto Vezzani}
\address{LAGA - Universit\'e Paris 13 (UMR 7539), 99 av. Jean-Baptiste Cl\'ement, 93430 Villetaneuse, France}

\email{vezzani@math.univ-paris13.fr}

\begin{abstract}
 We prove the equivalence between the category $\RigDM_{\et}^{\eff}(K,\Q)$ of effective motives of rigid analytic varieties over a perfect complete non-archimedean field $K$ and the category $\RigDA_{\Frobet}^{\eff}(K,\Q)$ which is obtained by localizing the category of motives without transfers $\RigDA_{\et}^{\eff}(K,\Q)$ over purely inseparable maps. In particular, we obtain an equivalence between $\RigDM_{\et}^{\eff}(K,\Q)$ and  $\RigDA_{\et}^{\eff}(K,\Q)$ in the characteristic $0$ case and an equivalence between $\DM_{\et}^{\eff}(K,\Q)$ and $\DA_{\Frobet}^{\eff}(K,\Q)$ of motives of algebraic varieties over a perfect field $K$. We also show a relative and a stable version of the main statement.
\end{abstract}

\maketitle

\tableofcontents

\section{Introduction}

Morel and Voevodsky in \cite{mv-99}  introduced the derived category of effective motives over a base $B$ which, in the abelian context with coefficients in a ring $\Lambda$ and with respect to the \'etale topology, is denoted by $\DA^{\eff}_{\et}(B,\Lambda)$. It is obtained as the homotopy category of the model category $\Ch\Psh(\Sm/B,\Lambda)$ of complexes of presheaves of $\Lambda$-modules over the category of smooth varieties over $B$, after a localization with respect to \'etale-local maps (giving rise to \'etale descent in homology) and  projection maps $\A^1_X\ra X$ (giving rise to the homotopy-invariance of homology). Voevodsky in  \cite{voe-tri},  \cite{mvw} also defined   the category of \emph{motives with transfers} $\DM^{\eff}_{\et}(B,\Lambda)$ using analogous constructions starting from the category $\Ch\PST(\Sm/B,\Lambda)$ of complexes of presheaves \emph{ with transfers} over $\Sm/B$ i.e. with extra functoriality with respect to maps which are finite and surjective. 
Both categories of motives can be stabilized, by formally inverting the Tate twist functor $\Lambda(1)$ in a model-categorical sense, giving rise to the categories of stable motives with and without transfers $\DM_{\et}(B,\Lambda)$ and $\DA_{\et}(B,\Lambda)$  respectively.

There exists a natural adjoint pair between the category of motives without and with transfers which is induced by the functor $a_{\tr}$ of ``adjoining transfers'' and its right adjoint $o_{\tr}$ of ``forgetting transfers''. Different authors have proved interesting results on the comparison between the two categories $\DA^{\eff}_{\et}(B,\Lambda)$ and $\DM^{\eff}_{\et}(B,\Lambda)$ induced by this adjunction. 
Morel in \cite{morel-rat} proved the equivalence between the stable categories $\DA_{\et}(B,\Lambda)$ and $\DM_{\et}(B,\Lambda)$ in case $\Lambda$ is a $\Q$-algebra and $B$ is the spectrum of a perfect field, by means of algebraic $K$-theory.  
Cisinski and Deglise in \cite{cd} generalized this fact to the case of a $\Q$-algebra $\Lambda $ and a base $B$ that is of finite dimension, noetherian, excellent and geometrically unibranch.  
Later, Ayoub  (see \cite[Theorem B.1]{ayoub-h1}) gave a simplified  proof of this equivalence for a normal basis $B$ in characteristic $0$ and a coefficient ring $\Lambda$ over $\Q$ that also works for the effective categories. In \cite{ayoub-etale} the same author proved the equivalence between the stable categories of motives with and without transfers for a more general ring of coefficients $\Lambda$, under some technical assumptions on the base $B$ (see \cite[Theorem B.1]{ayoub-etale}). 

The purpose of this paper is to give a generalization of the effective result of Ayoub \cite[Theorem B.1]{ayoub-h1}. We prove an equivalence between the effective categories of motives with rational coefficients for a  normal base $B$ over a perfect field $K$ of arbitrary characteristic. Admittedly, in order to reach this equivalence in characteristic $p$ we need to consider a perfect base $B^{\Perf}$ and invert extra maps in $\DA^{\eff}_{\et}(B^{\Perf},\Q)$ namely the purely inseparable morphisms, or equivalently the relative Frobenius maps.  This procedure can also be interpreted as a localization with respect to a finer topology, that we will call the $\Frobet$-topology. The associated homotopy category will be denoted by $\DA^{\eff}_{\Frobet}(B^{\Perf},\Q)$.

We remark that the approach \emph{without transfers} is much more convenient when computing morphisms, and it is the most natural over a general base. On the other hand, Voevodsky proved a series of useful theorems for the category of motives \emph{with transfers} over a field  (say, the Cancellation theorem \cite{voe-canc} or the homotopy invariance of cohomology \cite[Proposition 3.1.11]{voe-tri}) which are fundamental for developing the theory. Being able to switch between the two definitions via a canonical equivalence is  then useful when dealing with motives, and has been used intensively in the literature (see \cite{ayoub-ICM} for an overview). This article shows that one can finally do so also for effective motives in positive characteristic.

All our statements will be given in the setting of rigid analytic varieties instead of algebraic varieties. The reason is twofold: on the one hand one can deduce immediately the statements on algebraic motives by considering a trivially valued field, on the other hand comparison theorems for motives of rigid analytic varieties $\RigDA^{\eff}_{\et}(B,\Lambda)$ and $\RigDM^{\eff}_{\et}(B,\Lambda)$ are equally useful for some purposes. For example, the result in characteristic $0$ is mentioned and used in \cite[Section 2.2]{ayoub-conj}. Also, this equivalence  in case $B$ is the spectrum of a perfect field of arbitrary characteristic plays a crucial role in \cite{vezz-fw} and actually constitutes the main motivation of this work.  
For the theory of rigid analytic spaces over non-archimedean fields, we refer to \cite{BGR}.

The main theorem of the paper is the following (Theorem \ref{DA=DMp-app}):
\begin{thm*}
	Let $\Lambda$ be a $\Q$-algebra and let $B$ be a  normal rigid  variety over a perfect, complete non-archimedean field $K$. The functor $a_{tr}$ induces an equivalence of triangulated categories:
	\[
	{\LL a_{tr}}\colon{\RigDA^{\eff}_{\Frobet}}(B^{\Perf},\Lambda )\cong{\RigDM^{\eff}_{\et}}(B^{\Perf},\Lambda).
	\]
\end{thm*}

The article is organized as follows. In Section \ref{frobtop} we introduce the $\Frobet$-topology on normal varieties and we prove some general properties it satisfies. In Section \ref{frobmot} we define the categories of motives that we are interested in, as well as other categories of motives which play an auxiliary role in the proof of the main result. In Section \ref{DADMsec} we finally outline the proof of the equivalence above.

\section{The {Frob}-topology}\label{frobtop}

We first define a topology on normal rigid analytic varieties over a field $K$. Along our work, we will always assume the following hypothesis.

\begin{assu}
	We let $K$ be a  perfect  field  which is 
	complete with respect to a non-archimedean norm. 
\end{assu}

Unless otherwise stated, we will use the term ``variety'' to indicate a  
separated rigid analytic variety over $K$ (see \cite[Chapter 9]{BGR}). 

\begin{dfn}
	A map $f\colon Y\ra X$ of varieties  over $K$ is called a \emph{$\Frob$-cover} if it is finite, surjective and for every affinoid $U$ in $X$ the affinoid inverse image $V=f^{-1}(U)$ 
	is such that the induced map of rings $\mcO(U)\ra \mcO(V)$ is radicial.
\end{dfn}

\begin{rmk}
	By \cite[Corollary IV.18.12.11]{EGAIV4} a morphism of schemes is  finite, surjective and radicial if and only if it is a finite universal homeomorphism. We can remark that the same holds true for rigid analytic varieties. That said, we will not use this characterization in this text. 
\end{rmk}

If $\car K=p$ and  $X$ is a  variety over $K$ then the absolute $n$-th Frobenius map $X\ra X$ given by the elevation to the $p^n$-th power, factors over a map $X\ra X^{(n)}$ where we denote by $X^{(n)}$ the base change of $X$ by the absolute $n$-th Frobenius map $K\ra K$.  We denote by $\Phi^{(n)}$ the map $X\ra X^{(n)}$ and we call it the \emph{relative $n$-th Frobenius}. Since $K$ is perfect, $X^{(n)}$ is isomorphic to $X$ endowed with the structure map $X\ra \Spa K\stackrel{\Phi^{-n}}{\ra}\Spa K$ and the relative $n$-th Frobenius is isomorphic to the absolute $n$-th Frobenius of $X$ over $\F_p$. We can also  define $X^{(n)}$ for negative $n$  to be  the base change of $X$ over  the  map $\Phi^{n}\colon K\ra K$ which is again isomorphic to $X$ endowed with the structure map $X\ra \Spa K\stackrel{\Phi^{-n}}{\ra}\Spa K$. The Frobenius map induces a morphism $X^{(-1)}\ra X$ and the collection of maps $\{X^{(-1)}\ra X\}$ 
defines a coverage (see for example \cite[Definition C.2.1.1]{elephant1}).

In case $\car K=0$ we also define $X^{(n)}$ to be $X$ and the maps $\Phi\colon X^{(n-1)}\ra X^{(n)}$ to be the identity maps for all $n\in \Z$.

\begin{prop}\label{frobuniv}
	Let $Y\ra X$ be a $\Frob$-cover between  normal quasi-compact varieties over  $K$. There exists an integer $n$ and a map $X^{(-n)}\ra Y$ such that  the composite map $X^{(-n)}\ra Y\ra X$ coincides with $\Phi^{n}$ and the composite map $Y\ra X\ra Y^{(n)}$ coincides with $\Phi^{n}$. 
\end{prop}

\begin{proof}
	Let $f\colon Y\ra X$ a $\Frob$-cover of affinoid normal schemes over $K$. We can consider the induced map of $K$-algebras and apply \cite[Proposition 6.6]{kollar} to conclude that it exists an integer $n$ and a map $h\colon X\ra Y^{(n)}$ such that the composite map $Y\ra X\ra Y^{(n)}$ coincides with the relative $n$-th Frobenius. This factorization is also canonical, and therefore can be generalized to the situation in which $X$ and $Y$ are not necessarily affinoid.
	
	We also remark that the map $Y\ra X$ is an epimorphism (in the categorical sense) of normal varieties. 
	From the equalities $fhf^{(n)}=\Phi^{(n)}_Yf^{(n)}=f\Phi^{(n)}_X $ we then conclude that the composite map $X\ra Y^{(n)}\ra X^{(n)}$ coincides with the $n$-relative Frobenius. This proves the  claim. 
\end{proof}

\begin{dfn}
	Let $B$ be a normal variety over $K$. We define $\RigSm/B $ to be the category of  quasi-compact varieties which are smooth over $B$.  We denote by $\tau_{\et}$ the \'etale topology. 
\end{dfn}

\begin{dfn}
	Let $B$ be a normal variety over $K$. We define $\RigNor/B $ to be the category of quasi-compact normal varieties over $B$. 
	\begin{itemize}
		\item We denote by $\tau_{\Frob}$ the topology on $\RigNor/B$ induced by $\Frob$-covers. 
		\item We denote by $\tau_{\et}$ the \'etale topology. 
		\item We denote by $\tau_{\Frobet}$ the topology generated by $\tau_{\Frob}$ and  $\tau_{\et}$.
		\item We denote by $\tau_{\fh}$ the topology generated by covering families $\{f_i\colon X_i\ra X\}_{i\in I}$ such that $I$ is finite, and the induced map $\bigsqcup f_i\colon\bigsqcup_{i\in I} X_i\ra X$ is finite and surjective. 
		\item We denote by $\tau_{\fhet}$ the topology generated by $\tau_{\fh}$ and  $\tau_{\et}$. 
	\end{itemize}
\end{dfn}

\begin{rmk}
	The $\Frobet$ topology is denoted by $\quiet$ (quasi-\'etale) in \cite[Section 5]{fj} and the $\fhet$-topology is often denoted by $\qfh$ (see \cite{voe-h}). We stick to the notation $\fhet$ in order to be consistent with \cite{ayoub-h1}.
\end{rmk}

We are not imposing any additivity condition on the $\Frob$-topology, i.e. the families $\{X_i\ra\bigsqcup_{i\in I} X_i\}_{i\in I}$  are not $\Frob$-covers. This does not interfere much with our theory  since we will mostly be interested in the $\Frobet$-topology, with respect to which such families are covering families.

\begin{rmk}\label{sametopos}
	We are ultimately interested in considering the $\Frobet$-topology on $\RigNor/B$. As any object $X\in \RigNor/B$ is locally affinoid, we can restrict to considering the full subcategory $\Aff\!\Nor/B$ of $\RigNor/B$ made of \emph{affinoid} varieties that are smooth over $B$ since it induces an equivalent \'etale (and $\Frobet$) topos. In proofs we will then, sometimes tacitly, assume that the objects of $\RigNor/B$ and $\RigSm/B$ are affinoid, without loss of generality. For the same reason, one can harmlessly drop the condition on quasi-compactness for objects in $\RigNor/B$ and $\RigSm/B$ without changing the associated topoi.
\end{rmk}

\begin{rmk}
	The $\fh$-topology is obviously finer that the $\Frob$-topology, which is the trivial topology in case $\car K=0$.
\end{rmk}

\begin{rmk}
	The category of normal affinoid varieties is not closed under fiber products, and the $\fh$-coverings do not define a Grothendieck pretopology. Nonetheless, they define a coverage which is enough to have a convenient description of the topology they generate (see for example \cite[Section C.2.1]{elephant1}).  
\end{rmk}

\begin{rmk}\label{pseudoG}
	A particular example of $\fh$-covers is given by \emph{pseudo-Galois covers} which are finite, surjective maps $f\colon Y\ra X$ of normal integral affinoid varieties such that the field extension $K(Y)\ra K(X)$ is obtained as a composition of a Galois extension and a finite, purely inseparable extension. The Galois group $G$ associated to the extension coincides with $\Aut(Y/X)$.  
	As shown in \cite[Corollary 2.2.5]{ayoub-rig}, a presheaf $\mcF$ on $\Aff\!\Nor/B$ with values in a complete and cocomplete category is an $\fh$-sheaf if and only if the two following conditions are satisfied.
	\begin{enumerate}
		\item For every finite set $\{X_i\}_{i\in I}$ of objects in $\RigNor/B$ it holds $\mcF(\bigsqcup_{i\in I} X_i)\cong\prod_{i\in I}\mcF(X_i)$.
		\item For every pseudo-Galois covering $Y\ra X$ with associated Galois group $G$ the map $\mcF(X)\ra\mcF(Y)^G$ is invertible.
	\end{enumerate}
\end{rmk}

\begin{dfn}\label{RigSmBPerf}
	Let $B$ be a normal variety over $K$. 
	\begin{itemize}
		\item We denote by $\RigSm/B^{\Perf}$ the 2-limit category $2-\varinjlim_n\RigSm/B^{(-n)}$ with respect to the  functors $\RigSm/B^{(-n-1)}\ra\RigSm/B^{(-n)}$ induced by the pullback along the map $B^{(-n-1)}\ra B^{(-n)}$. More explicitly, it is equivalent to the category $\cat_B[S^{-1}]$ where $\cat_B$ is the category whose objects are pairs $(X,-n)$ with $n\in\N$ and $X\in\RigSm/B^{(-n)}$   and morphisms $\cat_B((X,-n),(X',-n'))$ are maps $f\colon X\ra X'$ forming  commutative squares
		$$\xymatrix{
			X\ar[d]\ar[r]^-{f} & X'\ar[d]\\
			B^{(-n)}\ar[r]^-{\Phi}&B^{(-n')}
		}$$
		and where $S$ is the class of canonical maps $(X'\times_{B^{(-n')}}B^{(-n)},-n)\ra(X',-n')$ for each $X\in\RigSm/B^{(-n')}$ and $n\geq n'$ (see \cite[Definition VI.6.3]{EGAIV2}).
		\item We say that a map $(X,-n)\ra(X',-n')$ of $\RigSm/B^{\Perf}$ is a\emph{ $\Frob$-cover} if the map $X\ra X'$ is a $\Frob$-cover. We denote by $\tau_{\Frob}$ the topology on $\RigSm/B^{\Perf}$ induced by $\Frob$-covers. 
		\item  		We denote by $\tau_{\et}$ the topology on $\RigSm/B^{\Perf}$ generated by the \'etale coverings on each category $\RigSm/B^{(-n)}$. It defines the ``inverse limit'' topology on $\RigSm/B^{\Perf}$ according to  \cite[Definition VI.8.2.5]{SGAIV2}.
		\item We denote by $\tau_{\Frobet}$ the topology generated by $\tau_{\Frob}$ and $\tau_{\et}$.
	\end{itemize}
\end{dfn}

We now investigate some properties of the $\Frob$-topology.

\begin{prop}\label{frobcov}
	Let $B$ be a normal variety over $K$. 
	\begin{itemize}
		\item A presheaf $\mcF$ on $\RigNor/B$   is a ${\Frob}$-sheaf if and only if $\mcF(X^{(-1)})\cong\mcF(X)$ for all  objects $X$ in $\RigNor/B$.
		\item A presheaf $\mcF$ on  $\RigSm/B^{\Perf}$ is a ${\Frob}$-sheaf if and only if   $\mcF(X^{(-1)},-n-1)\cong\mcF(X,-n)$ for all  objects $(X,-n)$ in $\RigSm/B^{\Perf}$.
	\end{itemize}
\end{prop}

\begin{proof}
	The two statements are analogous and  we only prove the claim for $\RigNor/B$. By means of \cite[Lemma C.2.1.6 and Lemma C.2.1.7]{elephant1} the topology generated by maps $f\colon Y\ra X$ which factor a power of Frobenius $X^{(-n)}\ra X$ is the same as the one generated by the coverage $X^{(-1)}\ra X$. Using Proposition \ref{frobuniv}, we conclude that the $\Frob$-topology coincides with the one generated by the coverage $\{X^{(-1)}\ra X\}$. 
	Since the Frobenius map is a monomorphism of normal varieties, the sheaf condition associated to the coverage $X^{(-1)}\ra X$ is simply the one of the statement by \cite[Lemma 2.1.3]{elephant1}.
\end{proof}

\begin{cor}\label{perfsm}
	Let $B$ be a normal variety over $K$.
	\begin{itemize}
		\item The class $\Phi$ of maps $\{ X^{(-r)}\ra X\}_{r\in\N,X\in\RigNor/B}$ admits calculus of fractions, and its saturation consists of  $\Frob$-covers. In particular, the continuous map \[(\RigNor/B,\Frob)\ra\RigNor/B[\Phi^{-1}]\] defines an equivalence of topoi.
		\item  The class $\Phi$ of maps $\{ (X^{(-r)},-n-r)\ra (X,-n)\}_{r\in\N,(X,n)\in\RigSm/B^{\Perf}}$ admits calculus of fractions, and its saturation consists of $\Frob$-covers. In particular, the continuous map  \[(\RigSm/B^{\Perf},\Frob)\ra\RigSm/B^{\Perf}[\Phi^{-1}]\] defines an equivalence of topoi.
	\end{itemize}
\end{cor}

\begin{proof}We only prove the first claim. 
	The fact that $\Phi$ admits calculus of fractions is an easy check, and the characterization of its saturation follows from Proposition \ref{frobuniv}. The sheaf condition for a presheaf $\mcF$ with respect to the $\Frob$-topology is simply $\mcF(X^{(-1)})\cong\mcF(X)$ by Corollary \ref{frobcov} hence the last claim.
\end{proof}

\begin{rmk}\label{addtransfer}
	We follow the notations introduced in Definition \ref{RigSmBPerf}. Any pullback of a finite, surjective radicial map between normal algebraic varieties is also finite, surjective and radicial. This can be generalized to  rigid analytic varieties, given the explicit description of the pull-back of a finite map (see for example \cite[Lemma 1.4.5]{huber}). 
	In particular, if $B$ is a normal variety, the maps in the class $S$  are invertible in $\RigNor/B[\Phi^{-1}]$. The functor $\cat_B\ra\RigNor/B[\Phi^{-1}]$ defined by mapping $(X,-n)$ to $X$ factors through a functor $\RigSm/B^{\Perf}\ra\RigNor/B[\Phi^{-1}]$. In particular, there is a functor $\RigSm/B^{\Perf}[\Phi^{-1}]\ra \RigNor/B[\Phi^{-1}]$ defined by sending $(X,-n)$ to $X$ hence, by  Corollary \ref{perfsm}, there is a functor $\Sh_{\Frob}(\RigSm/B^{\Perf})\ra\Sh_{\Frob}(\RigNor/B)$.
\end{rmk}

\begin{rmk}\label{e*}
	If $e\colon B'\ra B$ is a finite map of normal varieties, any \'etale hypercover $\mcU\ra B'$ has a refinement by a hypercover $\mcU'$  obtained by pullback from  an \'etale hypercover $\mcV$ of $B$ (see for example \cite[Tag 04DL]{stacks-project}).  In particular, the functor $e_*\colon\Psh(\RigSm/B')\ra\Psh(\RigSm/B)$ commutes with the functor $a_{\et}$ of $\et$-sheafification. The same holds true for the functor $e_*\colon\Psh(\RigSm/B'^{\Perf})\ra\Psh(\RigSm/B^{\Perf})$.
\end{rmk}

From now on, we fix a commutative ring $\Lambda$ and work  with $\Lambda$-enriched categories. In particular, the term ``presheaf'' should be understood as ``presheaf of $\Lambda$-modules'' and similarly for the term ``sheaf''. It follows that the presheaf $\Lambda(X)$ represented by an object $X$ of a category $\cat$ sends an object $Y$ of $\cat$ to the free $\Lambda$-module $\Lambda\Hom(Y,X)$. 

\begin{assu}
	Unless otherwise stated, we assume from now on that $\Lambda$ is a $\Q$-algebra and we omit it from the notations.
\end{assu}

The following facts are immediate, and will also be useful afterwards.

\begin{prop}\label{afrobet}
	Let $B$ be a normal variety over $K$. 
	\begin{itemize}
		\item If $\mcF$ is an \'etale sheaf on $\RigSm/B^{\Perf}$ [resp. on $\RigNor/B$] then $a_{\Frob}\mcF$ is a $\Frobet$-sheaf.
		\item If $\mcF$ is a $\Frob$-sheaf on $\RigSm/B^{\Perf}$ [resp. on $\RigNor/B$] then $a_{\et}\mcF$ is a $\Frobet$-sheaf.
	\end{itemize}
\end{prop}

\begin{proof}
	We only prove the claims for $\RigNor/B$. First, suppose that $\mcF$ is an \'etale sheaf. 
	By Proposition \ref{frobuniv}, we obtain that $a_{\Frob}\mcF(X)=\varinjlim_n\mcF(X^{(-n)})$. Whenever $U\ra X$ is \'etale, then $U\times_XX^{(-n)}\cong U^{(-n)}$ and $U^{(-n)}\times_{X^{(-n)}}U^{(-n)}\cong(U\times_XU)^{(-n)}$ so that the following diagram is exact
	\[
	0\ra\mcF(X^{(-n)})\ra\mcF(U^{(-n)})\ra\mcF((U\times_X U)^{(-n)}).
	\]
	The first claim the follows by taking the limit over $n$.
	
	We now prove the second claim. Suppose $\mcF$ is a $\Frob$-sheaf. For any \'etale covering $\mcU\ra X$ we indicate with $\mcU'$ the associated covering of $X^{(-1)}$ obtained by pullback. From Remark \ref{e*} one can compute the sections of $a_{\et}\mcF(X^{(-1)})$ with the formula
	\[
	a_{\et}\mcF(X^{(-1)})=\varinjlim_{\mcU\ra X}\ker\left(\mcF(\mcU'_0)\ra\mcF(\mcU'_1)\right)
	\]
	where $\mcU\ra X$ varies among \v{C}ech covers of $X$. Since $\mcF$ is a $\Frob$-sheaf, then $\mcF(U'_0)\cong\mcF(U_0)$ and $\mcF(U_1')\cong\mcF(U_1)$. The formula above then implies
	\[
	a_{\et}\mcF(X^{(-1)})=\varinjlim_{\mcU\ra X}\ker\left(\mcF(\mcU_0)\ra\mcF(\mcU_1)\right)=a_{\et}\mcF(X)
	\]
	proving the claim.
\end{proof}

\begin{prop}\label{fhet}
	Let $B$ be a normal variety over $K$. 
	If $\mcF$ is a $\fh$-sheaf on $\RigNor/B$ then $a_{\et}\mcF$ is a $\fhet$-sheaf.
\end{prop}

\begin{proof}
	Let $f\colon X'\ra X$ be a pseudo-Galois cover in $\Aff\!\Nor/B$  with associated group $G$. In light of Remark \ref{pseudoG}, we need to show that $a_{\et}\mcF(X)\cong a_{\et}\mcF(X')^G$. For any \'etale covering $\mcU\ra X$ we indicate with $\mcU'$ the associated covering of $X'$ obtained by pullback. From Remark \ref{e*} one can compute the sections of $a_{\et}\mcF(X')$ with the formula
	\[
	a_{\et}\mcF(X')=\varinjlim_{\mcU\ra X}\ker\left(\mcF(\mcU'_0)\ra\mcF(\mcU'_1)\right)
	\]
	where $\mcU\ra X$ varies among \v{C}ech covers of $X$. Taking the $G$-invariants is an exact functor as $\Lambda$ is a $\Q$-algebra and when applied to the formula above it yields
	\[
	a_{\et}\mcF(X')^G=\varinjlim_{\mcU\ra X}\ker\left(\mcF(\mcU'_0)^G\ra\mcF(\mcU'_1)^G\right)=\varinjlim_{\mcU\ra X}\ker\left(\mcF(\mcU_0)\ra\mcF(\mcU_1)\right)=a_{\et}\mcF(X)
	\]
	as wanted.
\end{proof}

\begin{prop}\label{oexact}
	Let $B$ be a  normal variety over $K$. The canonical inclusions 
	\[
	\begin{aligned}
	o_{\Frob}\colon&\Sh_{\Frob}(\RigNor/B)\ra\Psh(\RigNor/B)\\
	o_{\Frob}\colon&\Sh_{\Frob}(\RigSm/B^{\Perf})\ra\Psh(\RigSm/B^{\Perf})\\
	o_{\fh}\colon&\Sh_{\fh}(\RigNor/B)\ra\Psh(\RigNor/B)\\
	\end{aligned}
	\] 
	are exact.
\end{prop}

\begin{proof}
	In light of Proposition \ref{frobcov} the statements about $o_{\Frob}$ are obvious. Since $\Lambda$ is a $\Q$-algebra, the functor of $G$-invariants from $\Lambda[G]$-modules to $\Lambda$-modules is exact. The third claim then follows from Remark \ref{pseudoG}.
\end{proof}

We now investigate the functors of the topoi introduced above induced by a map of varieties $B'\ra B$. 

\begin{prop}\label{quillenpairsS}
	Let $f\colon B'\ra B$ be a map of normal varieties over $K$. 
	\begin{itemize}
		\item Composition with $f$ defines a functor $f_\sharp$ from normal varieties over $B'$ to normal varieties over $B$ which induces the following adjoint pair
		\[\adj{ f_\sharp}{\Ch\Sh_{\Frobet}(\RigNor/B')}{\Ch\Sh_{\Frobet}(\RigNor/B)} { f^*}\]
		\item The base change over $f$ defines functors  $ f^{(-n)*}$ from  smooth varieties over $B^{(n)}$ to  smooth varieties over $B'^{(n)}$  which induce the following adjoint pair
		\[\adj{f^*}{\Ch\Sh_{\Frobet}(\RigSm/B^{\Perf})}{\Ch\Sh_{\Frobet}(\RigSm/B'^{\Perf})}{f_*}\]
		\item If $f$ is a $\Frob$-cover,  the functors above are equivalences of categories.
		\item If $f$ is a  smooth map,  the composition with $f$ defines functors $f^{(-n)}_\sharp$ from smooth varieties over $B'^{(-n)}$ to smooth varieties over $B^{(-n)}$ which induce the following adjoint pair
		\[\adj{f_\sharp}{\Ch\Sh_{\Frobet}(\RigSm/B'^{\Perf})}{\Ch\Sh_{\Frobet}(\RigSm/B^{\Perf})}{f^*}\]
	\end{itemize}
\end{prop}

\begin{proof}
	We initially remark that the functors $f^{(-n)*}$ induce a functor $f^*\colon \cat_B\ra\cat_{B'}$ where $\cat_B$ is the fibered category introduced in Definition \ref{RigSmBPerf} where we drop the condition of being quasi-compact (see Remark \ref{sametopos}). As  cartesian squares are mapped to cartesian squares, they also induce a functor from smooth varieties over $B^{\Perf}$ to smooth varieties over $B'^{\Perf}$.
	
	The existence of the first two adjoint pairs is then a formal consequence of the continuity of the functors $f_\sharp$ and $f^*$.

	Let now $f$ be a $\Frob$-cover. The functors $f^*\colon\RigSm/B^{\Perf}[\Phi^{-1}]\ra\RigSm/B'^{\Perf}[\Phi^{-1}]$ and $f_\sharp\colon \RigNor/B'[\Phi^{-1}]\ra\RigNor/B[\Phi^{-1}]$ are equivalences, and we conclude the third claim by what proved above and Corollary \ref{perfsm}.
	
	For the  fourth claim, we  use a different model for the $\Frobet$-topos on $\RigSm/B^{\Perf}$.  
	The fibered category  $\cat_B$  can be endowed with the $\Frob$-topology and the $\Frobet$-topology. Following the proof of Corollary \ref{perfsm}, the map $(\cat_B,\Frob)\ra\cat_B[\Phi^{-1}]$ induces an equivalence of topoi. Moreover, the canonical functor $\cat_B[\Phi^{-1}]\ra\RigSm/B^{\Perf}[\Phi^{-1}]$  induces    an equivalence of categories. 
	
	The existence of the last Quillen functor is therefore a formal consequence of the continuity of the functor  $f_\sharp\colon(\cat_{B'}[\Phi^{-1}],\et)\ra (\cat_B[\Phi^{-1}],\et)$.
\end{proof}

\begin{rmk}
	Let $f\colon B'\ra B$ be a map of normal varieties. The image via $f^*$ of the presheaf represented by $(X,-n)$ is the  presheaf represented by $(X\times_BB'^{(-n)},-n)$ and if $f$ is smooth, the image via $f_\sharp$ of the presheaf represented by $(X',-n)$ is the sheaf represented by $(X',-n)$.
\end{rmk}

\section{Rigid motives and Frob-motives}\label{frobmot}

We recall that the ring of coefficients $\Lambda$ is assumed to be a $\Q$-algebra, and that presheaves and sheaves take values in the category of $\Lambda$-modules.

We make extensive use of the theory of model categories and localization, following the approach of Ayoub in \cite{ayoub-th2} and \cite{ayoub-rig}. Fix a  site $(\cat,\tau)$. 
The category of complexes of presheaves $\Ch(\Psh(\cat))$ can be endowed with the \emph{projective model structure} for which weak equivalences are quasi-isomorphisms (maps inducing isomorphisms of homology presheaves) and fibrations are maps $\mcF\ra\mcF'$ such that $\mcF(X)\ra\mcF'(X)$ is a surjection for all $X$ in $\cat$ (cfr \cite[Section 2.3]{hovey} and \cite[Proposition 4.4.16]{ayoub-th2}). 

\begin{rmk}\label{modelcatch}
	If we take $\cat=\{*\}$ we obtain in particular the usual projective model category structure on $\Ch(\Lambda)$ which is cellular and left proper (see for example \cite[Example 4.4.24(2)]{ayoub-th2} and \cite[Proposition 2.3.22]{hovey}). For any $\cat$ the category $\Ch(\Psh(\cat))$ is equivalent to the category of presheaves on $\cat$ with values in $\Ch(\Lambda)$. With this respect, the projective model structure described above  coincides with the one induced by defining weak-equivalences and fibrations point-wise, starting from the projective model structure on $\Ch(\Lambda)$. One could alternatively consider the (Quillen equivalent) \emph{injective model structure} on $\Ch(\Psh(\cat))$ obtained by defining weak-equivalences and cofibrations point-wise (see \cite[Definition 4.4.15]{ayoub-th2}).
\end{rmk}

Also the category of complexes of sheaves $\Ch(\Sh_\tau(\cat))$ can be endowed with the \emph{projective model structure} defined in  \cite[Proposition 4.4.41]{ayoub-th2}. In this structure, weak equivalences are quasi-isomorphisms of complexes of sheaves (maps inducing isomorphisms on the sheaves associated to the homology presheaves).

Just as in \cite{jardine-s},  \cite{mvw}, \cite{mv-99} or \cite{riou}, we  consider the  left Bousfield localization of $\Ch(\Psh(\cat))$ with respect to the topology we select, and  a chosen ``contractible object''. We recall that left Bousfield localizations with respect to a class of maps $S$ (see \cite[Chapter 3]{hirschhorn}) is the universal model categories in which the maps in $S$ become weak equivalences. The existence of such structures is granted only under some technical hypothesis, as shown in \cite[Theorem 4.1.1]{hirschhorn} and \cite[Theorem 4.2.71]{ayoub-th2}.

\begin{prop}\label{locsets}
	Let $(\cat,\tau)$ be a site with finite direct products and let $\cat'$ be a full subcategory of $\cat$  such that every object of $\cat$ has a covering by objects of $\cat'$. Let also $I$ be an object of $\cat'$.
	\begin{enumerate}
		\item The projective model category $\Ch\Psh (\cat)$  admits a  left Bousfield localization $\Ch_{I}\Psh (\cat)$  with respect to the set $S_{I}$ of all maps $\Lambda(I\times X)[i]\ra\Lambda( X)[i]$ as $X$ varies in $\cat$  and $i$ varies in $\Z$. 
		\item The projective model categories $\Ch\Psh (\cat)$ and $\Ch\Psh(\cat')$  admit  left Bousfield localizations  $\Ch_{\tau}\Psh (\cat)$ and $\Ch_{\tau}\Psh(\cat')$ with respect to the class $S_{\tau}$   of   maps $\mcF\ra\mcF'$ inducing isomorphisms on the $\et$-sheaves associated to $H_i(\mcF)$ and $H_i(\mcF')$ for all $i\in\Z$. Moreover, the two localized model categories are Quillen equivalent and the sheafification functor induces a Quillen equivalence to the projective model category  $\Ch\Sh_{\tau}(\cat)$.
		\item The  model categories $\Ch_{\tau}\Psh (\cat)$ and $\Ch_{\tau}\Psh(\cat')$  admit  left Bousfield localizations $\Ch_{\tau,I}\Psh (\cat)$ and $\Ch_{\tau,I}\Psh(\cat')$ with respect to the set $S_{I}$ defined above.  Moreover, the two localized model categories are Quillen equivalent.
	\end{enumerate}
\end{prop}

\begin{proof}
	By \cite[Proposition 4.4.16]{ayoub-th2} and Remark \ref{modelcatch} the projective model structures  in the statement are  left proper and cellular. Any such model category admits a left Bousfield localization with respect to a set of maps ( \cite[Theorem 4.1.1]{hirschhorn}) hence the first claim.
	
	For the first part of second claim, it suffices to apply \cite[Proposition 4.4.32, Lemma 4.4.35]{ayoub-th2} showing that the localization over $S_\tau$ is equivalent to a localization over a set of maps. The second part is a restatement of  \cite[Corollary 4.4.43, Proposition 4.4.56]{ayoub-th2}.
	
	Since by  \cite[Proposition 4.4.32]{ayoub-th2} the $\tau$-localization coincides with the Bousfield localization with respect to a set, we conclude by \cite[Theorem 4.2.71]{ayoub-th2} that the  model category $\Ch_{\tau}\Psh (\cat)$ is still left proper and cellular. The last statement then follows from  \cite[Theorem 4.1.1]{hirschhorn} and the second claim.
\end{proof}

In the situation above, we will denote by $S_{(\tau,I)}$ the union of the class $S_{\tau}$ and the set $S_I$.

\begin{rmk}
	A geometrically relevant situation is induced when $I$ is endowed with a multiplication map $\mu\colon I\times I\ra I$ and maps $i_0$ and  $i_1$ from the terminal object to $I$ satisfying the relations of a monoidal object with $0$ as in the definition of an interval object (see \cite[Section 2.3]{mv-99}). Under these hypotheses, we say that the triple $(\cat,\tau,I)$ is a \emph{site with an interval}.
\end{rmk}

\begin{exm}
	The affinoid rigid variety  $\B^1=\Spa K\langle \chi\rangle$ is an interval object with respect to the natural multiplication $\mu$ and maps $i_0$ and  $i_1$ induced by the substitution  $\chi\mapsto0 $ and  $\chi\mapsto1$ respectively. 
\end{exm}

\begin{dfn}
	Let $B$ be a normal variety over $K$. 
	\begin{itemize}
		\item The triangulated homotopy category of $\Ch_{\et,\B^1}\Psh (\RigSm/B)$   will be denoted by $\RigDA_{\et}^{\eff}(B,\Lambda)$. 
		\item The triangulated homotopy category of $\Ch_{\et,\B^1}\Psh (\Rig\Sm/B^{\Perf})$   will be denoted by $\RigDA_{\et}^{\eff}(B^{\Perf},\Lambda)$ and  the one of $\Ch_{\Frobet,\B^1}\Psh (\Rig\Sm/B^{\Perf})$   will be denoted by $\RigDA_{\Frobet}^{\eff}(B^{\Perf},\Lambda)$. 
		\item The triangulated homotopy category of $\Ch_{\Frobet,\B^1}\Psh (\RigNor/B)$   will be denoted by $\catD_{\Frobet,\B^1}(\RigNor/B,\Lambda)$ and the one of  $\Ch_{\fhet,\B^1}\Psh (\RigNor/B)$   will be denoted by  $\catD_{\et,\B^1}^{\fh}(\RigNor/B,\Lambda)$.  
		\item If $\cat$ is one of the categories $\RigSm/B$, $\Rig\Sm/B^{\Perf}$ and $\RigNor/B$ and $\eta\in\{\et,\Frob,\fh,\Frobet,\fhet, \B^1,(\et,\B^1),(\Frobet,\B^1),(\fhet,\B^1)\}$  we say that a map   in  $\Ch\Psh (\cat)$  is a \emph{$\eta$-weak equivalence} if it is a weak equivalence in the model structure   $\Ch_{\eta}\Psh (\cat)$ whenever this makes sense.
		\item  We will omit $\Lambda$ from the notation whenever the context allows it. The image of a variety $X$ in one of these categories will be denoted by $\Lambda(X)$. 
	\end{itemize}
\end{dfn}

We now want to introduce the analogue of the previous definitions for motives with transfers. 
By Remark \ref{addtransfer} the mapping $(X,-n)\mapsto X$ induces a functor $\Sh_{\Frob}(\RigSm/B^{\Perf})\ra\Sh_{\Frob}(\RigNor/B)$. If we compose it with the Yoneda embedding and the functor $a_{\fh}$ of $\fh$-sheafification we obtain a functor
\[
\RigSm/B^{\Perf}\ra\Sh_{\Frob}(\RigSm/B^{\Perf})\ra\Sh_{\fh}(\RigNor/B).
\]

\begin{dfn} 
	Let $B$ be a normal variety over $K$.
	\begin{itemize}
		\item We define the category $\RigCor/B$ as the category whose objects are those of $\RigSm/B$ and whose morphisms $\Hom(X,Y)$ are computed in $\Sh_{\fh}(\RigNor/B)$.
		The category $\Psh(\RigCor/B)$ will be denoted by $\PST(\RigSm/B)$.
		\item We define the category $\RigCor/B^{\Perf}$ as the category whose objects are those of $\RigSm/B^{\Perf}$ and whose morphisms $\Hom(X,Y)$ are computed in $\Sh_{\fh}(\RigNor/B)$.
		The category $\Psh(\RigCor/B^{\Perf})$ will be denoted by $\PST(\RigSm/B^{\Perf})$.
	\end{itemize}
\end{dfn}

We remark that, as $\Lambda$ is a $\Q$-algebra, morphisms $X\ra Y$ of $\RigCor$ admit a more concrete description in terms of \emph{correspondences} defined in \cite[Noltation 2.2.22]{ayoub-rig} and denoted in \cite{ayoub-rig} by $\Cor(X,Y)$. 
We also remark that the inclusions of categories $\RigSm/B\ra\RigCor/B$ and $\RigSm/B^{\Perf}\ra\RigCor/B^{\Perf}$ induce the following adjunctions:
\[
\adj{a_{\tr}}{\Ch\Psh(\RigSm/B)}{\Ch\PST(\RigSm/B)}{o_{\tr}}.
\]
\[
\adj{a_{\tr}}{\Ch\Psh(\RigSm/B^{\Perf})}{\Ch\PST(\RigSm/B^{\Perf})}{o_{\tr}}.
\]

We now define the category of motives with transfers.

\begin{prop}\label{rigdmperfet}
	Let $B$ be a normal variety and $\cat$ be either $\RigSm/B$ or $ \RigSm/B^{\Perf}$. 
	The projective model  category $\Ch\PST (\cat)$   
	admits a left Bousfield localization $\Ch_{\et}\PST (\cat)$   with respect to $S_{\et}$, the class  of  of maps $f$ such that $o_{\tr}(f)$ is a $\et$-weak equivalence. 
	It also admits a further Bousfield localization $\Ch_{\et,\B^1}\PST (\cat)$   with respect to the set  formed by all maps $\Lambda(\B^{1}_X)[i]\ra\Lambda( X)[i]$ by letting $X$ vary in $\cat$  and $i$ vary in $\Z$.
\end{prop}

\begin{proof}
	The proof of \cite[Theorem 2.5.7]{ayoub-rig} also applies in our situation. For the second statement, it suffices to apply \cite[Theorem 4.1.1]{hirschhorn}.
\end{proof}

\begin{rmk}
	By means of an \'etale version of \cite[Corollary 2.5.3]{ayoub-rig}, if $\mcF$ is a presheaf with transfers then the associated \'etale sheaf $a_{\et}\mcF$ can be endowed with a unique structure of presheaf with transfers such that $\mcF\ra a_{\et}\mcF$ is a map of presheaves with transfers. The class $S_{\et}$ can then be defined intrinsically, as the class of maps $\mcF\ra\mcF'$ inducing isomorphisms of \'etale sheaves with transfers $a_{\et}H_i\mcF\ra a_{\et}H_i\mcF'$. 
\end{rmk}

\begin{dfn}
	Let $B$ be a normal variety over $K$. 
	\begin{itemize}
		\item The triangulated homotopy category  of $\Ch_{\et,\B^1}\PST (\RigSm/B)$ will be denoted by $\RigDM_{\et}^{\eff}(B,\Lambda)$. 
		\item The triangulated homotopy category  of $\Ch_{\et,\B^1}\PST (\RigSm/B^{\Perf})$ will be denoted by $\RigDM_{\et}^{\eff}(B^{\Perf},\Lambda)$.
		\item  We will omit $\Lambda$ from the notation whenever the context allows it. The image of a variety $X$ in one of these categories will be denoted by $\Lambda_{\tr}(X)$. 
	\end{itemize}
\end{dfn}

We remark that if $\car K=0$ the two definitions above coincide. Also, if $B$ is the spectrum of the perfect field $K$ the category $\RigDM^{\eff}_{\et}(B^{\Perf} )$ coincides with $\RigDM^{\eff}_{\et}(K )$. In this case, the definition of $\RigDA_{\Frobet}^{\eff}(B^{\Perf} )$ also coincides with the one of $\RigDA_{\Frobet}^{\eff}(K)$ given in the introduction 
as the following fact shows.

\begin{prop}\label{frobetlocal}
	Let $B$ be a  normal variety over $K$. The category   $\Ch_{\Frobet}( \RigSm/B^{\Perf} )$ is Quillen equivalent  to the left Bousfield localization of $\Ch_{\et}\Psh(\RigSm/B^{\Perf})$ over the set of all shifts of maps $\Lambda(X^{(-1)},-n-1)\ra\Lambda( X,-n)$ as $(X,-n)$ varies in $\RigSm/B^{\Perf}$. 
\end{prop}

\begin{proof}
	From Lemmas \ref{afrobet}, \ref{oexact} and \ref{toptop}   we conclude that $\Frobet$-local objects are those which are $\Frob$-local and $\et$-local. We can then conclude using Lemma \ref{froblocal}.
\end{proof}

\begin{lemma}\label{toptop}
	Let $\cat$ be a category endowed with two Grothendieck topologies $\tau_1$, $\tau_2$ and let $\tau_3$ be the topology generated by $\tau_1$ and $\tau_2$. We denote by $a_{\tau_i}$ 
	the associated sheafification functor and with $o_{\tau_i}$ their right adjoint functors. If $o_{\tau_1}$ is exact and   
	$a_{\tau_3}=a_{\tau_2}a_{\tau_1}$ then the  following  categories are   canonically equivalent:
	\begin{enumerate}
		\item The homotopy category of $\Ch_{\tau_3}\Psh(\cat )$.
		\item The full triangulated subcategory of $\catD(\Psh(\cat ))$ formed by objects which are $\tau_3$-local.
		\item The full triangulated subcategory of $\catD(\Psh(\cat ))$ formed by objects which are $\tau_1$-local and $\tau_2$-local.
	\end{enumerate}
\end{lemma}

\begin{proof}
	The equivalence between the first and the second category 
	follows by definition of the Bousfield localization.  We are left to prove the equivalence between the second and the third. We remark that $\tau_3$-local objects are in particular $(\tau_1,\tau_2)$-local. 
	
	Since $o_{\tau_1}$ is exact, the category of $\tau_1$-local objects coincides with the category of complexes quasi-isomorphic to complexes of $\tau_1$-sheaves. 
	Consider the model category $\Ch_{\tau_3}(\Sh_{\tau_1}(\cat))$ which is the Bousfield localization of $\Ch(\Sh_{\tau_1}(\cat))$ over the class of maps   of complexes inducing isomorphisms on the $\tau_3$-sheaves associated to the homology presheaves, that we will call $\tau_3$-equivalences. From the assumption $a_{\tau_3}=a_{\tau_2}a_{\tau_1}$ the class of $\tau_3$-equivalences coincides with the class of maps $S_{\tau_2}$  of complexes inducing isomorphisms on the $\tau_2$-sheaves associated to the homology $\tau_1$-sheaves. Hence $\Ch_{\tau_3}(\Sh_{\tau_1}(\cat))$ coincides with $\Ch_{\tau_2}(\Sh_{\tau_1}(\cat))$ and its derived category is equivalent to the category of $(\tau_1,\tau_2)$-local complexes.

	Because of the following Quillen adjunction
	\[
	\adj{\LL a_{\tau_1}=a_{\tau_1}}{\Ho(\Ch_{\tau_3}\Psh(\cat)}{\Ho(\Ch_{\tau_3}\Sh_{\tau_1}(\cat))}{\RR o_{\tau_1}=o_{\tau_1}}.
	\]
	we conclude that the image via $o_{\tau_1}$ of a $\tau_2$-local complex of sheaves i.e. a $(\tau_1.\tau_2)$-local complex, is $\tau_3$-local, as wanted.
\end{proof}

\begin{lemma}\label{froblocal}
	Let $B$ be a  normal variety over $K$.  A fibrant object of $\Ch\Psh(\RigSm/B^{\Perf})$ is $\Frob$-local if and only if it is local with respect to the set of all shifts of maps $\Lambda(X^{(-1)},-n-1)\ra\Lambda(X,-n)$ as $(X,-n)$ varies in $\RigSm/B^{\Perf}$. 
\end{lemma}

\begin{proof} 
	We initially remark that a fibrant complex $\mcF$ is local with respect to the set of maps in the claim if and only if $(H_i\mcF)(X,-n)\cong (H_i\mcF)(X^{(-1)},-n-1)$ for all $X$ and $i$. By Proposition \ref{frobuniv}, this amounts to say that $H_i\mcF$ is a $\Frob$-sheaf for all $i$.  
	
	Suppose now that  $\mcF$ is fibrant and $\Frob$-local. Since the map of presheaves $\Lambda(X^{(-1)},-n-1)\ra\Lambda(X,-n)$ induces an isomorphism on the associated $\Frob$-sheaves, we deduce that $(H_i\mcF)(X^{(-1)},-n-1)\cong(H_i\mcF)(X,-n)$. This implies that $H_i\mcF$ is a $\Frob$-sheaf and hence $\mcF$ is local with respect to the maps of the claim, as wanted.
	
	Suppose now that  $\mcF$ is fibrant and local with respect to the maps of the claim. Let $\mcF\ra C^{\Frob}\mcF$ a $\Frob$-weak equivalence to a fibrant $\Frob$-local object. By definition, we deduce that the $\Frob$-sheaves associated to $H_i\mcF$ and to $H_iC^{\Frob}\mcF$ are isomorphic. On the other hand, we know that these presheaves are already $\Frob$-sheaves, and hence the map $\mcF\ra C^{\Frob}\mcF$ is a quasi-isomorphism of presheaves and $\mcF$ is $\Frob$-local.
\end{proof}

We now want to find another model for the category $\catD^{\fh}_{\et,\B^1}(\RigNor/B)$. This is possible by means of the model-categorical machinery developed above.

By Remark \ref{pseudoG} an object $\mcF$ in $\Ch\Psh(\RigNor/B )$ is $\fh$-local if and only if it is additive and 
\[
{\catD\Psh(\RigNor/B)}(\Lambda(X),\mcF)\ra{\catD\Psh(\RigNor/B)}(\Lambda(X'),\mcF)^{\Aut(X'/X)}
\]
is an isomorphism, for all pseudo-Galois coverings $X'\ra X$ in $\Aff\!\Nor/B$. 
Therefore,  if we consider $\catD_{\Frobet,\B^1}(\RigNor/B)$ as the subcategory of $(\B^1,\Frobet)$-local objects in $\catD\Psh(\RigNor/B)$ we say that  an object $\mcF$ of $\catD_{\Frobet,\B^1}(\RigNor/B ) $ is $\fh$-local if and only if 
\[
{\catD_{\Frobet,\B^1}(\RigNor/B)}(\Lambda(X),\mcF)\ra{\catD_{\Frobet,\B^1}(\RigNor/B)}(\Lambda(X'),\mcF)^{\Aut(X'/X)}
\]
is an isomorphism, for all pseudo-Galois coverings $X'\ra X$.

\begin{prop}\label{compcor}
	Let $B$ be a  normal variety over $K$. The category $\catD^{\fh}_{\et,\B^1}(\RigNor/B )$ is canonically isomorphic  to the 
	category of $\fh$-local objects in $\catD_{\Frobet,\B^1}(\RigNor/B )$.
\end{prop}

\begin{proof}
	It suffices to prove the claim before performing the $\B^1$-localization on each category. The statement then follows from  Propositions \ref{afrobet} and \ref{fhet} together with Lemmas \ref{oexact} and \ref{toptop}.
\end{proof}

We now study some functoriality properties of the categories just defined, and later prove a fundamental fact: the {locality axiom} (see \cite[Theorem 3.2.21]{mv-99}).

\begin{prop}\label{quillenpairs}
	Let $f\colon B'\ra B$ be a map of normal varieties over $K$. 
	The first two adjoint pairs of Proposition \ref{quillenpairsS}  
	induce the following Quillen pairs: 
	\[\adj{\LL f_\sharp}{\catD_{\Frobet,\B^1}(\RigNor/B')}{\catD_{\Frobet,\B^1}(\RigNor/B)}{\RR f^*}\]
	\[\adj{\LL f^*}{\RigDA_{\Frobet}^{\eff}(B^{\Perf})}{\RigDA_{\Frobet}^{\eff}(B'^{\Perf})}{\RR f_*}\]
	which are equivalences whenever $f$ is a $\Frob$-covering. Moreover, 
	if $f$ is a  smooth map,  the third adjoint pair of Proposition \ref{quillenpairsS} induces 
	a Quillen pair:
	\[\adj{\LL f_\sharp}{\RigDA_{\Frobet}^{\eff}(B'^{\Perf})}{\RigDA_{\Frobet}^{\eff}(B^{\Perf})}{\LL f^*}\]
\end{prop}
\begin{proof}
	The statement is a formal consequence of Proposition \ref{quillenpairsS}, \cite[Theorem 4.4.61]{ayoub-th2} and  the formulas $f^*(\B^1_X)=\B^1_{f^*(X)}$ and $f_\sharp(\B^1_X)=\B^1_X$. 
\end{proof}

\begin{prop}\label{Ri*}
	Let $e\colon B'\ra B$ be a finite map of normal varieties over $K$. The functor 
	\[e_*\colon \Ch\Psh(\RigSm/B'^{\Perf})\ra \Ch\Psh(\RigSm/B^{\Perf})\]
	preserves the $(\Frobet,\B^1)$-equivalences.
\end{prop}

\begin{proof}
	Let $e\colon B'\ra B$ be a finite map of normal varieties. The functor $e_*$ is induced by the map $\RigSm/B^{\Perf}\ra\RigSm/B'^{\Perf}$ sending $(X,-n)$ to $(X\times_{B^{(-n)}}B'^{(-n)},-n)$. From Remark \ref{e*} it commutes with $\et$-sheafification. As the image of $(X^{(-1)},-n-1)$ is isomorphic to $((X\times_{B^{(-n)}}B'^{(-n)})^{(-1)},-n-1)$ we deduce from Corollary \ref{perfsm} that $e_*$ commutes with $\Frob$-sheafification. Therefore by Proposition \ref{afrobet} we deduce that 
	$e_*\colon\Psh(\RigSm/B'^{\Perf})\ra\Psh(\RigSm/B^{\Perf})$ commutes with the functor $a_{\Frobet}$ of $\Frobet$-sheafification, hence it preserves $\Frobet$-equivalences.

	We now prove that it also preserves $\B^1$-equivalences. By \cite[Proposition 4.2.74]{ayoub-th2} it suffices to show that $e_*(\Lambda(\B^1_V)\ra\Lambda(V))$ is a $\B^1$-weak equivalence for any $V$ in $\RigSm/X'^{\Perf}$. This follows from the explicit homotopy between the identity and the zero map on $e_*(\Lambda(\B^1_V))$ (see the argument of \cite[Theorem 2.5.24]{ayoub-rig}).  
\end{proof}

The following property  is an extension of \cite[Theorem 1.4.20]{ayoub-rig} and referred to as the \emph{locality axiom.} 

\begin{thm}\label{locality}
	Let $i\colon Z\hookrightarrow B$ be a closed immersion of normal varieties over $K$ and let $j\colon U\hookrightarrow B$ be the open complement. For every object $M$ in $\RigDA^{\eff}_{\Frobet}(B^{\Perf})$ there is 
	an distinguished triangle
	\[
	\LL j_\sharp\LL j^* M\ra M\ra \RR i_*\LL i^*M\ra
	\]
	In particular, the pair $(\LL j^*,\LL i^*)$ is conservative.
\end{thm}

\begin{proof}
	First of all, we remark that by Proposition \ref{Ri*} one has $\RR i_*=i_*$. In particular it suffices to prove the claim before performing the localization over the shifts of maps $\Lambda(X^{(-1)},-n-1)\ra\Lambda( X,-n)$ 
	i.e. in the  category  $\RigDA_{\et}^{\eff}(B^{\Perf} )$. 
	
	The functors  $\LL j_\sharp$ $\LL j^*$  and $\LL i^*$ commute with small sums because they admit right adjoint functors. Also $\RR i_*$ does, since it holds $\RR i_*=i_*$.  We conclude that  the full subcategory  of $\RigDA^{\eff}_{\Frobet}(B^{\Perf})$ of objects $M$ such that 
	\[
	\LL j_\sharp\LL j^* M\ra M\ra \RR i_*\LL i^*M\ra
	\]
	is an distinguished triangle is closed under cones, and under small sums. We can then equivalently prove the claim in the subcategory $\RigDA^{\ct}_{\et}(B^{\Perf})$ of compact objects, since these motives generate $\RigDA^{\eff}_{\et}(B^{\Perf})$ as a triangulated category with small sums.
	
	By means of Lemma \ref{DA=limDA} and Proposition \ref{Ri*}, it suffices to prove
	the statement for  each category $\RigDA^{\eff}_{\et}(B^{(-n)} )$. It is then enough to prove the claim for the categories $\RigDA^{\eff}_{\Nis}(B^{(-n)} )$ as defined in \cite[Definition 1.4.12]{ayoub-rig} since $\RigDA^{\eff}_{\et}(B^{(-n)} )$ is a further localization of $\RigDA^{\eff}_{\Nis}(B^{(-n)} )$. In this case, the statement is proved in  \cite[Theorem 1.4.20]{ayoub-rig}.
\end{proof}

\begin{lemma}\label{DA=limDA} Let $B$ be a  normal variety over $K$. 
	The canonical functors $\RigSm/B^{(-n)}\ra\RigSm/B^{\Perf}$ induce a triangulated equivalence of categories \[\varinjlim_n\RigDA^{\ct}_{\et}(B^{(-n)} )\cong\RigDA^{\ct}_{\et}(B^{\Perf} )\]
	where we denote by $\RigDA^{\ct}_{\et}(B^{(-n)} )$ [resp.  with $\RigDA^{\ct}_{\et}(B^{\Perf} )$] the subcategory  of compact objects of $\RigDA^{\eff}_{\et}(B^{(-n)} )$ [resp.  of $\RigDA^{\eff}_{\et}(B^{\Perf} )$].
\end{lemma}

\begin{proof}
	The functor $\RigDA^{\eff}_{\et}(B^{(-n)} )\ra\RigDA^{\eff}_{\et}(B^{\Perf} )$ is triangulated and sends the objects $\Lambda(X)[i]$ which are compact generators of the first category,  to compact objects of the second. It then induces an exact functor between the two subcategories of compact objects. Moreover, by letting $n$ vary, the images of the objects in $\RigDA^{\ct}_{\et}(B^{(-n)} )$ generate the category $\RigDA^{\ct}_{\et}(B^{\Perf} )$.

	Up to shifting indices, it therefore suffices to show that for $X$, $Y$ in $\RigSm/B$  one has  
	\[\varinjlim_n\RigDA^{\eff}_{\et}(B^{(-n)} )(\Lambda(X\times_BB^{(-n)}),\Lambda(Y\times_BB^{(-n)}))\cong\RigDA^{\eff}_{\et}(B^{\Perf} )(\Lambda(\bar{X}),\Lambda(\bar{Y}))\]
	where we denote by $\bar{X}=(X,0)$ and $\bar{Y}=(Y,0)$ the object of $\RigSm/B^{\Perf}$ associated to $X$ resp. $Y$. 
	To this aim, we simply follow the proof of \cite[Proposition 1.A.1]{ayoub-rig}. For the convenience of the reader, we reproduce it here. 
	
	\textit{Step 1}: We consider the  directed diagram $\mcB$ formed the maps $B^{(-n-1)}\ra B^{(-n)}$ and we let $\RigSm/\mcB$ be the  the category of
	rigid smooth varieties over it as defined in \cite[Section 1.4.2]{ayoub-rig}. We can endow the category $\Ch\Psh(\RigSm/\mcB)$ with the $(\et,\B^1)$-local model structure, and consider the Quillen adjunctions induced by the map of diagrams $\alpha_n\colon B^{(-n)}\ra\mcB$, $f_{nm}\colon B^{(-n)}\ra B^{(-m)}$:
	\[
	\begin{aligned}
	\adj{\alpha_n^*}{\Ch\Psh(\RigSm/\mcB)}{\Ch\Psh(\RigSm/{B^{(-n)}})}{\alpha_{n*}}\\
	\adj{\alpha_{n\sharp}}{\Ch\Psh(\RigSm/{B^{(-n)}})}{\Ch\Psh(\RigSm/\mcB)}{\alpha_{n}^*}\\
	\adj{f_{nm}^*}{\Ch\Psh(\RigSm/{B^{(-m)}})}{\Ch\Psh(\RigSm/{B^{(-n)}})}{f_{nm*}}
	\end{aligned}
	\]
	We also remark that 
	the canonical map $\RigSm/{B^{(-n)}}\ra\RigSm/B^{\Perf}$ induces a Quillen adjunction
	\[
	\adj{f_{\infty n}^*}{\Ch\Psh(\RigSm/{B^{(-n)}})}{\Ch\Psh(\RigSm/{B^{\Perf}})}{f_{\infty n*}}.
	\]
	Consider a trivial cofibration $ \alpha_{0*}\Lambda(Y)\ra R$ with target $R$ that is $(\et,\B^1)$-fibrant. Since $\alpha^*_{n}$ is a left and right Quillen functor and $\alpha^*_n\alpha_{0*}=f_{n0}^*$ we deduce that the map $\Lambda(Y\times_BB^{(-n)})=f_{n0}^*\Lambda(Y)\ra \alpha_n^*R $ is also an $(\et,\B^1)$-trivial cofibration with an $(\et,\B^1)$-fibrant target.

	\textit{Step 2}: By applying the left Quillen functors $f_{nm}^*$ and $f_{\infty m}^*$ we also obtain that $f_{n0}^*\Lambda(Y)=f_{nm}^*f_{m0}^*\Lambda(Y)\ra f_{nm}^*\alpha_m^*R$ and $f_{\infty0}^*\Lambda(Y)=f_{\infty m}^*f_{m0}^*\Lambda(Y)\ra f_{\infty m}^*\alpha_m^*R $ are $(\et,\B^1)$-trivial cofibrations. By the 2-out-of-3 property of weak equivalences applied to the composite map 
	\[
	f_{n0}^*\Lambda(Y)\ra f_{nm}^*\alpha_m^*R\ra \alpha_n^*R
	\]
	we then deduce that the map $f_{nm}^*\alpha_m^*R\ra \alpha_n^*R$ is an $(\et,\B^1)$-weak equivalence. 
	
	\textit{Step 3}: We now claim that the natural map $\Lambda(\bar{Y})\rightarrow \hat{R}$ with $\hat{R}:= \colim_nf_{\infty n}^*\alpha_i^*R$ is an $(\et,\B^1)$-weak equivalence in $\Ch\Psh(\RigSm/{B^{\Perf}})$. By what shown in Step 2, it suffices to prove that the functor \[\colim\colon \Ch\Psh(\RigSm/{B^{\Perf}})^{\N}\ra\Ch\Psh(\RigSm/{B^{\Perf}})\] preserves $(\et,\B^1)$-weak equivalences. First of all, we remark that it is a Quillen left functor with respect to the projective model structure on the diagram category  $\Ch\Psh(\RigSm/{B^{\Perf}})^{\N}$ induced by the point-wise $(\et,\B^1)$-structure. Hence, it preserves $(\et,\B^1)$-weak equivalences between cofibrant objects. On the other hand, as directed colimits commute with homology, it also preserves weak equivalences of presheaves. Since any complex is quasi-isomorphic to a cofibrant one, we deduce the claim.
	
	\textit{Step 4}: We now prove that $\hat{R}$ is $\B^1$-local. Consider a variety $U$ smooth over $B^{(-n)}$. 
	From the formula
	\[
	\hat{R}(\bar{U})=\colim_{m\geq n}\alpha^*_{m}R(U\times_{B^{(-n)}}B^{(-m)})
	\]
	and the fact that $\alpha^*_mR$ is $\B^1$-local, we deduce a quasi-isomorphism $\hat{R}(U)\cong\hat{R}(\B^1_U)$ as wanted.
	
	\textit{Step 5}: We now prove that $\hat{R}$ is $\et$-local. It suffices to show that for any   $U$ smooth over $B^{(-n)}$  one has  $\HH^k_{\et}(\bar{U},\hat{R})\cong H_{-k}\hat{R}(\bar{U})$. 
	The topos associated to $\Et/U$ is equivalent to the one of  $\varinjlim\Et/(U\times_{B^{(-n)}}B^{(-m)})$ and all these sites have a bounded cohomological dimension since $\Lambda$ is a $\Q$-algebra. 
	By applying 
	\cite[Theorem VI.8.7.3]{SGAIV2} together with a spectral sequence argument given by \cite[Theorem 0.3]{sv-bk}, we then deduce the formula
	\[
	\HH^k_{\et}(\bar{U},\hat{R})\cong\colim_m\HH^k_{\et}(U\times_{B^{(-n)}}B^{(-m)},\alpha_m^*R). 
	\]
	On the other hand, as $\alpha_m^*R$ is $\et$-local, we conclude that
	\[\colim_m\HH^k_{\et}(U\times_{B^{(-n)}}B^{(-m)},\alpha_i^*R)\cong\colim_mH_{-k}(\alpha_m^*R)(U\times_{B^{(-n)}}B^{(-m)})\cong H_{-k} \hat{R} (\bar{U})
	\]
	proving the claim.
	
	\textit{Step 6}: From Steps 3-5, we conclude that we can compute ${\RigDA^{\eff}_{\et}(B^{\Perf)})}(\Lambda(\bar{X}),\Lambda(\bar{Y}))$ as $\hat{R}(\bar{X})$ which coincides with $\colim_n(\alpha_n^*R)(X\times_{B}B^{(-n)})$. By what is proved in Step 1, we also deduce that $\alpha_n^*R$ is a $(\et,\B^1)$-fibrant replacement of $\Lambda(Y\times_{B}B^{(-n)})$ and hence the last group coincides with $\colim_n{\RigDA^{\eff}_{\et}(B^{(-n)})} (\Lambda(X\times_{B}B^{(-n)}),\Lambda({Y\times_{B}B^{(-n)}}))$ proving the statement.
\end{proof}

\section{The equivalence between motives with and without transfers}\label{DADMsec}

We can finally present the main result of this paper. We recall that the ring of coefficients $\Lambda$ is assumed to be a $\Q$-algebra.

\begin{thm}\label{DA=DMp-app}
	Let $B$ be a  normal variety over $K$. The functor $a_{tr}$ induces an equivalence of triangulated categories:
	\[
	{\LL a_{tr}}\colon{\RigDA^{\eff}_{\Frobet}}(B^{\Perf} )\cong{\RigDM^{\eff}_{\et}}(B^{\Perf} ).
	\]
\end{thm}

As a corollary, we obtain the two following results, which are indeed our main motivation.

\begin{thm}\label{DA=DMp-appK}
	The functor $a_{tr}$ induces an equivalence of triangulated categories:
	\[
	{\LL a_{tr}}\colon{\RigDA^{\eff}_{\Frobet}}(K )\cong{\RigDM^{\eff}_{\et}}(K ).
	\]
\end{thm}

\begin{thm}\label{DA=DM-app}
	Let $B$ be a normal  variety over a field $K$ of characteristic $0$. The functor $a_{tr}$ induces an equivalence of triangulated categories:
	\[
	{\LL a_{tr}}\colon{\RigDA^{\eff}_{\et}}(B )\cong{\RigDM^{\eff}_{\et}}(B ).
	\]
\end{thm}

The proof of Theorem \ref{DA=DMp-app} is divided into the following steps.

\begin{enumerate}
	\item We first produce a functor $\LL a_{\tr}\colon{\RigDA^{\eff}_{\Frobet}}(B^{\Perf} )\ra{\RigDM^{\eff}_{\et}}(B^{\Perf} )$ commuting with sums, triangulated, sending a set of compact generators of the first category into a set of compact generators of the second.
	\item We define a fully faithful functor $\LL i^*\colon {\RigDA^{\eff}_{\Frobet}}(B^{\Perf} )\ra\catD_{\Frobet,\B^1}^{\fh}(\RigNor/B )$.
	\item We define a fully faithful functor $\LL j^*\colon {\RigDM^{\eff}_{\et}}(B^{\Perf} )\ra\catD_{\Frobet,\B^1}^{\fh}(\RigNor/B )$.
	\item We check that $\LL j^*\circ\LL a_{\tr}$ is isomorphic to $\LL i^*$ proving that $\LL a_{\tr}$ is also fully faithful.
\end{enumerate}

We now prove the first step.  

\begin{prop}\label{step1a}
	Let $B$ be a normal variety over $K$. 
	The functor $a_{\tr}$ induces a triangulated functor \[\LL a_{\tr}\colon {\RigDA^{\eff}_{\Frobet}}(B^{\Perf})\ra{\RigDM^{\eff}_{\et}}(B^{\Perf})\] commuting with sums, sending a set of compact generators of the first category into a set of compact generators of the second.
\end{prop}

\begin{proof}
	The functor $a_{\tr}$ induces a Quillen functor \[\LL a_{\tr}\colon\Ch_{\et}\Psh(\RigSm/B^{\Perf})\ra\Ch_{\et}\PST(\RigSm/B^{\Perf})\] sending $\Lambda(X,-n)$ to $\Lambda_{\tr}(X)$. We are left to prove that it  factors over the $\Frob$-localization, i.e. that the map $\Lambda_{\tr}(X^{(-1)})\ra\Lambda_{\tr}( X)$ is an isomorphism in $\RigDM^{\eff}_{\et}(B^{\Perf})$ for all $X\in\RigSm/B^{(-n)}$. Actually, since the map $X^{(-1)}\ra X$ induces an isomorphism of $\fh$-sheaves, we deduce that it is an isomorphism in the category $\RigCor/B^{\Perf}$ hence also in $\RigDM^{\eff}_{\et}(B^{\Perf})$.
\end{proof}

We are now ready to prove the second step.

\begin{prop}\label{B.2}
	Let $B$ be a  normal variety over $K$. The functors $\RigSm/B^{(-n)}\ra\RigNor/B$ 
	induce a fully faithful functor $$\LL i_B^*\colon {\RigDA^{\eff}_{\Frobet}(B^{\Perf} )}\ra\catD_{\Frobet,\B^1}(\RigNor/B ).$$
\end{prop}

\begin{proof}
	We let $\cat_B$ be the category introduced in Definition \ref{RigSmBPerf}. As already remarked in the proof of Proposition \ref{quillenpairsS} we can endow it with the $\Frobet$-topology and the topos associated to it is equivalent to the $\Frobet$-topos on $\RigSm/B^{\Perf}$. In particular, the continuous functor $i_B\colon\cat_B\ra\RigNor/B$ induces an adjunction 
	\[
	\adj{\LL i_B^*}{\RigDA^{\eff}_{\Frobet}(B^{\Perf})}{\catD_{\Frobet,\B^1}(\RigNor/B)}{\RR i_{B*}}.
	\]
	As $i_{B*}i_B^*$ is isomorphic to the identity, it suffices to show that $\R i_{B*}=i_{B*}$ so that $\R i_{B*}\LL i^*_B$ is isomorphic to the identity as well. 
	
	The functor $i_{B*}$ commutes with $\Frobet$-sheafification, and hence it preserves $\Frobet$-weak equivalences, and since $i_{B*}(\Lambda(\B^1_V))\cong  \Lambda(\B^1_B)\otimes i_{B*}(\Lambda(V))$ is weakly equivalent to $i_{B*}(\Lambda(V))$ for every $V$ in $\RigNor/B$ we also conclude that it preserves $\B^1$-weak equivalences, as wanted.
\end{proof}

\begin{rmk}\label{i*p}
	As a corollary of the proof of Proposition \ref{B.2} we obtain that the functor $i_{B*}$ preserves $(\Frobet,\B^1)$-equivalences.
\end{rmk}

We remark that the previous result does not yet prove our claim. This is reached by the following crucial fact. Its proof will demand a series of technical lemmas that are proven right below it.

\begin{prop}\label{arefhlocal}
	Let $B$ be a normal variety over $K$. 
	The image of $\LL i_B^*$ is contained in the subcategory of  $\fh$-local objects. 
\end{prop}

\begin{proof}
	Let $M$ be an object of $\RigDA^{\eff}_{\Frobet}(B^{\Perf} )$ let $f\colon X\ra B$ be a normal irreducible variety over $B$ and let $r\colon X'\ra X$ be a pseudo-Galois covering in $\Aff\!\Nor/B$ with $G=\Aut(X'/X)$.  We are left to prove that 
	\[
	{\catD_{\Frobet,\B^1}(\RigNor/B )}(\Lambda(X),\LL i^*M)\ra{\catD_{\Frobet,\B^1}(\RigNor/B )}(\Lambda(X'),\LL i^*M)^G
	\]
	is an isomorphism. Using Lemma \ref{B.5} we can equally prove that
	\[
	{\RigDA^{\eff}_{\Frobet}(X^{\Perf} )}(\Lambda,\LL f^*M)\ra{\RigDA^{\eff}_{\Frobet}(X'^{\Perf} )}(\Lambda,\LL r^*\LL f^*M)^G
	\]
	is an isomorphism. Using the notation of Lemma \ref{B.7}, it suffices to prove that the natural transformation $\id\ra(\R r_*\LL r^*)^G$ is invertible.
	
	Using Lemma \ref{strat}, we can define a stratification $(X_i)_{0\leq i\leq n}$ of $X$ made of locally closed connected normal subvarieties of $X$ such that $r_i\colon X'_i\ra X_i$ is a composition of an \'etale cover and a $\Frob$-cover 
	of normal varieties,
	by letting $X'_i$ be 
	the reduction of
	the subvariety $X_i\times_XX'\subset X'$. Using the locality axiom (Theorem \ref{locality}) for $\RigDA^{\eff}_{\Frobet}$ applied to the inclusions $u_i\colon X_i\ra X$ we can then restrict to proving that each transformation $\LL u_i^*\ra\LL u_i^*(\R r_*\LL r^*)^G\cong (\R r_{i*}\LL r_i^*)^G\LL u_i^*$ is invertible, where the last isomorphism follows from Lemma \ref{B.7}. It suffices then to prove that $\id\ra(\R r_{i*}\LL r_i^*)^G$ is invertible. 
	If $s\colon Z\ra T$ is a $\Frob$-cover, the functors $(\LL s^*,\RR s_*)$ define an equivalence of categories  $\RigDA^{\eff}_{\Frobet}(T^{\Perf} )\cong \RigDA^{\eff}_{\Frobet}(Z^{\Perf} )$ by Proposition \ref{quillenpairs} hence we can assume that the maps $r_i$ are \'etale covers. 
	Moreover, since $\LL r_i^*\colon\RigDA^{\eff}_{\Frobet}(X_i^{\Perf} )\ra\RigDA^{\eff}_{\Frobet}(X_i'^{\Perf} )$ is conservative by Lemma \ref{conservative}, we can equivalently prove that $\LL r_i^*\ra\LL r_i^*(\R r_{i*}\LL r_i^*)^G\cong (\R r'_{i*}\LL r_i'^*)^G\LL r_i^*$ is invertible, where $r_i'$ is the   base change of $r_i$ over itself (see Lemma \ref{B.7}). By the assumptions on $r_i$ we conclude that $r_i'$ is a projection $\bigsqcup X'_i\ra X'_i$ with $G$ acting transitively on the fibers, so that the functor $(\R r'_{i*}\LL r_i'^*)^G$ is the identity, proving the claim.
\end{proof}

The following lemmas were used in the proof of the previous proposition.

\begin{lemma}\label{B.5}
	Let $f\colon B'\ra B$ be a map of normal rigid varieties over $K$. For any $M\in\RigDA_{\Frobet}(B )$ there is a canonical isomorphism
	\[
	{\catD_{\Frobet,\B^1}(\RigNor/B)}(\Lambda(B'),\LL i^*_BM)\cong{\RigDA_{\Frobet}(B')}(\Lambda,\LL f^*M).
	\]
\end{lemma}

\begin{proof}
	Consider  the following diagram of functors:
	$$\xymatrix{
		\Psh(\cat_B[\Phi^{-1}])\ar[r]^-{i_B^*}\ar[d]^{f^*} & \Psh(\RigNor/B[\Phi^{-1}])\ar[d]^{f^{*}}\\
		\Psh(\cat_{B'}[\Phi^{-1}])\ar[r]^-{i_{B'}^*}&\Psh(\RigNor/B'[\Phi^{-1}])
	}$$
	Let $\mcF$ be in $\Psh(\cat_B[\Phi^{-1}])$ and $X'$ be in $\RigNor/B'$.  One has  $(i_{B'}^*f^*) (\mcF)(X')=\colim \mcF(V)$ where the colimit is taken over the  maps $X'\ra V\times_{B^{(-n)}}B'^{(-n)}$ in $\RigNor/B'[\Phi^{-1}]$ by letting $V$ vary among varieties which are smooth over some $B^{(-n)}$. On the other hand,  one has  $(f^*i_B^*)(\mcF)(X')=\colim\mcF(V)$ where the colimit is taken over the  maps $X'\ra V$ in $\RigNor/B[\Phi^{-1}]$ by letting $V$ vary among varieties which are smooth over some $B^{(-n)}$. Since 
	$V\times_{B^{(-n)}}B'^{(-n)}\cong(V\times_BB')_{\red}$ in $\RigSm/B'[\Phi^{-1}]$ we deduce that the indexing categories are equivalent, hence  
	the diagram above is commutative and therefore by Corollary \ref{perfsm} and what shown in the proof of Proposition \ref{quillenpairsS} also the following one is: 
	$$\xymatrix{
		\Ch\Sh_{\Frobet}(\RigSm/B^{\Perf})\ar[r]^-{i_B^*}\ar[d]^{f^*} & \Ch\Sh_{\Frobet}(\RigNor/B)\ar[d]^{f^{*}}\\
		\Ch\Sh_{\Frobet}(\RigSm/B'^{\Perf})\ar[r]^-{i_{B'}^*}&\Ch\Sh_{\Frobet}(\RigNor/B')
	}$$
	This fact  together with Lemma \ref{124} implies  $ f^{*}\LL i^*_B\cong\LL i^*_{B'} \LL f^{*}$. By Propositions \ref{quillenpairs} and \ref{B.2} we then deduce
	\[
	\begin{aligned}
	&{\catD_{\Frobet,\B^1}(\RigNor/B)}(\Lambda(B'),\LL i^*_BM)= {\catD_{\Frobet,\B^1}(\RigNor/B)}(\LL f_\sharp(\Lambda),\LL i^*_BM)\cong \\ &\cong  
	{\catD_{\Frobet,\B^1}(\RigNor/B')}(\Lambda, f^{*}\LL i^*_BM)\cong  {\catD_{\Frobet,\B^1}(\RigNor/B')}(\Lambda,\LL i^*_{{B'}}\LL f^{*}M)\cong \\
	&\cong{\catD_{\Frobet,\B^1}(\RigNor/B')}(\LL i^*_{B'}\Lambda,\LL i^*_{{B'}}\LL f^{*}M)\cong {\RigDA_{\Frobet}(B')}(\Lambda,\LL f^{*}M)
	\end{aligned}
	\]
	as claimed.
\end{proof}

\begin{lemma}\label{124}
	Let $f\colon B'\ra B$ be a map of normal varieties over $K$. The functor \[f^*\colon\Ch\Psh(\RigNor/B)\ra\Ch\Psh(\RigNor/B')\] preserves the $(\Frobet,\B^1)$-equivalences.
\end{lemma}

\begin{proof}
	Since $f^*$ commutes with $\Frobet$-sheafification and with colimits, it preserves $\Frobet$-equivalences. Since $f^*(\Lambda(\B^1_V))\cong  \B^1_B\otimes f^*(\Lambda(V))$ is weakly equivalent to $f^*(\Lambda(V))$ for every $V$ in $\RigNor/B$ we also conclude  that $f^*$ preserves $\B^1$-weak equivalences, hence the claim.
\end{proof}

\begin{lemma}\label{conservative}
	Let $B$ be a  normal variety over $K$ and let $f\colon X\ra Y$ be a composition of  $\Frob$-coverings and $\et$-coverings in $\RigNor/B$. The functor $\LL f^*\colon\RigDA^{\eff}_{\Frobet}(Y^{\Perf} )\ra\RigDA^{\eff}_{\Frobet}(X^{\Perf} )$ is conservative.
\end{lemma}

\begin{proof}
	If $f$ is a $\Frob$-cover, then  $\LL f^*$ is an equivalence by Proposition \ref{quillenpairs}.  We are left to prove the claim in case $f$ is an $\et$-covering. In this case, we can use the proof of the analogous statement in algebraic geometry \cite[Lemma 3.4]{ayoub-etale}.
\end{proof}

\begin{lemma}\label{B.7}
	Let $e\colon X'\ra X$ be a finite  morphism of normal  varieties over $K$ and let  $G$ be a finite  group acting on $\R e_*\LL e^*$. There exists a subfunctor $(\R e_*\LL e^*)^G$ of $\R e_*\LL e^*$ such that for all $M$, $N$ in $\RigDA^{\eff}_{\Frobet}(X^{\Perf} )$  one has  
	\[
	{\RigDA^{\eff}_{\Frobet}(X^{\Perf} )}(M,(\R e_*\LL e^*)^G N)\cong{\RigDA^{\eff}_{\Frobet}(X^{\Perf} )}(M,\R e_*\LL e^*N)^G.
	\]
	Moreover for any map $f\colon Y\ra X$ of normal rigid varieties 
	factoring into a closed embedding followed by a smooth map, and any diagram of normal varieties
	$$\xymatrix{
		(Y\times_XX')_{\red}\ar[r]^-{f'}\ar[d]^{e'}&X'\ar[d]^{e}\\
		Y\ar[r]^{f}&X
	}$$
	there is an induced action of $G$ on $\R e'_*\LL e'^*$ and an invertible transformation  $\LL f^*(\R e_*\LL e^*)^G\stackrel{\sim}{\ra}(\R e'_*\LL e'^*)^G\LL f^*$.
\end{lemma}

\begin{proof}
	We define $(\R e_*\LL e^*)^G$ to be subfunctor obtained as the image of the projector $\frac{1}{|G|}\sum g$ acting on $\R e_*\LL e^*$.
	
	In order to prove the second claim, it suffices to prove that $\LL f^*\RR e_*\LL e^*\cong \RR e'_*\LL e'^*\LL f^*$. As the latter term coincides with $\RR e'_*\LL(fe')^*=\RR e'_*\LL(ef')^*=\RR e'_*\LL f'^*\LL e^*$ it suffices to show that the base change transformation $\LL f^*\R e_*\ra \R e'_*\LL f'^*$  is invertible. We can consider individually the case in which $f$ is smooth, and the case in which $f$ is a closed embedding. 
	
	{\it Step 1}: Suppose that $f$ is smooth. Then $f^*$ has a left adjoint $f_\sharp$. We can equally prove that the natural transformation $\LL f'_\sharp \LL e'^*\ra \LL e^* \LL f_\sharp$ is invertible. This follows from the isomorphism between the functors $f'_\sharp e'^*$ and $e^* f_\sharp$ from $\Psh(\RigSm/X'^{\Perf})$ to $\Psh(\RigSm/Y^{\Perf})$ obtained by direct inspection.
	
	{\it Step 2}: Suppose that $f$ is a closed immersion. Let $j\colon U\ra X$ be the open immersion complementary to $f$ and $j'$ be the open immersion complementary to $f'$. By the locality axiom (Theorem \ref{locality}) we can equally prove that $\LL j_\sharp \R e'_*\ra \R e_* \LL j'_\sharp$ is invertible. 
	
	{\it Step 3}: It is easy to prove that the transformation $\LL j_\sharp \R e'_*\ra \R e_* \LL j'_\sharp$  is invertible once we know that $e_*$, $e'_*$, $j_\sharp$ and $j'_\sharp$ preserve the $(\Frobet,\B^1)$-equivalences.  
	Indeed, if this is the case, the functors derive trivially and it suffices to prove that for any $\Frobet$-sheaf $\mcF$ the map $(j_\sharp e'_*)(\mcF)\ra (e_* j'_\sharp)(\mcF)$ is invertible. This follows from the very definitions.
	
	{\it Step 4}: The fact that $j_\sharp$ (and similarly $j'_\sharp$) preserves the $(\Frobet)$-weak equivalences follows from the fact that it respects quasi-isomorphisms of complexes of $\Frobet$-sheaves, since it is the  functor of extension by $0$.  In order to prove that it preserves the $\B^1$-equivalences, by \cite[Proposition 4.2.74]{ayoub-th2} we can prove that $j_\sharp(\Lambda(\B^1_V)\ra\Lambda( V))$ is a $\B^1$-weak equivalence for all $V$ in $\RigSm/U^{\Perf}$ and this is clear. 
	The fact that $e_*$ (and similarly $e'_*$) preserves the $(\Frobet,\B^1)$-equivalences is proved in Proposition \ref{Ri*}. We then conclude the claim in case $f$ is a closed immersion.
\end{proof}

\begin{lemma}\label{strat}
	Let $f\colon X'\ra X$ be a pseudo-Galois map of normal  varieties over $K$. There exists a finite stratification $(X_i)_{1\leq i\leq n}$ of locally closed normal subvarieties of $X$ such that each induced map $f_i\colon (X'\times_XX_i)_{\red}\ra X_i$ is a composition of an \'etale cover and a  $\Frob$-cover of normal rigid varieties.
\end{lemma}

\begin{proof}
	For every affinoid rigid variety $\Spa R $ there is a map of ringed spaces $\Spa R\ra\Spec R$ which is surjective on points, and such that the pullback of a finite \'etale map $\Spec S\ra\Spec R$ [resp. of an open inclusion $U\ra\Spec R$] over $\Spa R\ra\Spec R$ exists (following the notation of \cite[Lemma 3.8]{huber2}) and is finite \'etale [resp. an open inclusion]. The claim then follows from the analogous statement valid for schemes over $K$.
\end{proof}

\begin{rmk}
	In the  proof of Proposition \ref{arefhlocal}, we made use of the fact that $\Lambda$ is a $\Q$-algebra in a crucial way, for instance, in order to define the functor $(\R e_*\LL e^*)^G$.
\end{rmk}

The following result proves the second step.

\begin{cor}Let $B$ be a normal variety over $K$. 
	The composite functor \[ {\RigDA^{\eff}_{\Frobet}}(B^{\Perf} )\ra\catD_{\Frobet,\B^1}(\RigNor/B )\ra \catD^{\fh}_{\et,\B^1}(\RigNor/B )\] is fully faithful. 
\end{cor}

\begin{proof}
	This follows at once from   Proposition \ref{compcor} and Proposition \ref{arefhlocal}.
\end{proof}

We now move to the third step. We recall that the category $\RigCor(B^{\Perf})$ is a subcategory of $\Sh_{\fh}(\RigNor/B)$. 
We denote by $j$ this inclusion of categories.

\begin{prop}
	Let $B$ be a  normal variety over $K$. The functor $j$ induces a fully faithful functor $\LL j^*\colon\RigDM^{\eff}(B^{\Perf} )\ra\catD^{\fh}_{\et,\B^1}(\RigNor/B )$.
\end{prop}

\begin{proof}
	The functor $j$ extends to a functor $\PST(\RigSm/B^{\Perf})\ra\Sh_{\fh}(\RigNor/B)$  and induces a Quillen pair $\adj{j^*}{\Ch\PST(\RigSm/B^{\Perf})}{\Ch\Sh_{\fh}(\RigNor/B)}{j_*}$ with respect to the projective model structures. We prove that it is a Quillen adjunction also with respect to the $(\et,\B^1)$-model structure on the two categories by showing that $j_*$ preserves  $(\et,\B^1)$-local objects. 
	
	From the following commutative diagram
	$$\xymatrix{
		\RigSm/B^{\Perf}\ar[r]\ar[d]	&	\Psh(\RigSm/B^{\Perf})\ar[d]^-{a_{\tr}}\ar[r]^-{i}	&	\Sh_{\Frob}(\RigNor/B)\ar[d]^-{a_{\fh}}\\ 
		\RigCor/B^{\Perf}\ar[r]	    	&	\PST(\RigSm/B^{\Perf})\ar[r]^{j}			&	\Sh_{\fh}(\RigNor/B)
	}$$
	we deduce that $o_{\tr}j_*=i_*o_{\fh}$ which is a right Quillen functor. It therefore suffices to show that if $o_{\tr}\mcF$ is  $(\et,\B^1)$-local then also $\mcF$ is, for every fibrant object $\mcF$. Let $\mcF\ra\mcF'$ be a $(\et,\B^1)$-weak equivalence to a $(\et,\B^1)$-fibrant object of $\Ch\PST(\RigSm/B^{\Perf})$. By Lemma \ref{2.111}, we deduce that $o_{\tr}\mcF\ra o_{\tr}\mcF'$ is a $(\et,\B^1)$-weak equivalence between $(\et,\B^1)$-fibrant objects, hence it is a quasi-isomorphism. As $o_{\tr}$ reflects quasi-isomorphisms, we conclude that $\mcF$ is quasi-isomorphic to $\mcF'$ hence $(\et,\B^1)$-local.
	
	We now prove that $\LL j^*$ is fully faithful by proving  that $\R j_*\LL j^*$ is isomorphic to the identity. As  $j_*j^*$ is isomorphic to the identity, it suffices to show that $\R j_*=j_*$. We start by proving that $j_*$ preserves $\Frobet$-weak equivalences. As shown in Remark \ref{i*p}, the functor $i_*$ preserves $\Frobet$-equivalences. It is also clear that $o_{\fh}$ does. Since $o_{\tr}$ reflects $\Frobet$-weak equivalences, the claim follows from the equality $o_{\tr}j_*=i_*o_{\fh}$. 
	Since  $j_*(\Lambda(\B^1_V))\cong  \Lambda(\B^1_B)\otimes j_*(\Lambda(V))$ is weakly equivalent to $j_*(\Lambda(V))$ for every $V$ in $\RigNor/B$, we also conclude that   $j_*$ preserves $\B^1$-weak equivalences, hence the claim.
\end{proof}

\begin{lemma}\label{2.111} 
	Let $B$ be a  normal variety over $K$. 
	The functor
	\[o_{\tr}\colon \Ch\PST(\RigSm/B^{\Perf})\ra\Ch\Psh(\RigSm/B^{\Perf})\]
	preserves   $(\et,\B^1)$-weak equivalences.
\end{lemma}

\begin{proof}
	The argument of \cite[Lemma 2.111]{ayoub-h1} easily generalizes to our context. We point out that in the proof, the  the class of \emph{injective} trivial cofibrations in the category of complexes of presheaves is used (see Remark \ref{modelcatch}).
\end{proof}

The fourth step is just an easy check, as the next proposition shows.

\begin{prop} Let $B$ be a  normal variety over $K$. 
	The composite functor $\LL j^*\circ\LL a_{\tr}$ is isomorphic to $\LL i^*$. In particular $\LL a_{\tr}$ is  fully faithful.
\end{prop}

\begin{proof}
	It suffices to check that the following square is quasi-commutative.
	$$\xymatrix{
		\Psh(\RigSm/B^{\Perf} )\ar[d]^{i}\ar[r]^{a_{\tr}}    &     \PST(\RigSm/B^{\Perf} )\ar[d]^{j}	\\
		\Sh_{\Frob}(\RigNor/B )\ar[r]^{a_{\fh}}		&   \Sh_{\fh}(\RigNor/B )
	}$$
	
	This can be done by inspecting the two composite right adjoints, which are canonically isomorphic.
\end{proof}

This also  ends the proof of Theorem \ref{DA=DMp-app}.

We remark that in case $K$ is endowed with the trivial norm, we obtain a result on the category of motives  constructed from  schemes over $K$. 
It is the natural generalization of \cite[Theorem B.1]{ayoub-h1} in positive characteristic. We recall that the ring of coefficients $\Lambda$ is assumed to be a $\Q$-algebra.

\begin{thm}\label{DA=DMp-standard}
	Let $B$ be a normal algebraic variety over a perfect field $K$. The functor $a_{tr}$ induces an equivalence of triangulated categories:
	\[
	{\LL a_{tr}}\colon{\DA^{\eff}_{\Frobet}}(B^{\Perf} )\cong{\DM^{\eff}_{\et}}(B^{\Perf} ).
	\]
\end{thm}

We now define the stable version of the categories of motives introduced so far, and remark that Theorem \ref{DA=DM-app} extends formally to the stable case providing a generalization of the result \cite[Theorem 15.2.16]{cd}. 

\begin{dfn}
	We denote by $\RigDA_{\Frobet}(B^{\Perf} )$ [resp. by $\RigDM_{\et}(B^{\Perf} )$] the homotopy category associated to the model category of symmetric spectra (see \cite[Section 4.3.2]{ayoub-th2}) 
	$ \Sp^\Sigma_{T}\Ch_{\Frobet,\B^1}\Psh(\RigSm/B^{\Perf})$ [resp.  $ \Sp^\Sigma_{T}\Ch_{\et,\B^1}\PST(\RigSm/B^{\Perf})$] where $T$ is the cokernel of the unit map $\Lambda(B)\ra\Lambda(\T^1_B)$ [resp $\Lambda_{\tr}(B)\ra\Lambda_{\tr}(\T^1_B)$]. 
\end{dfn}

\begin{cor}\label{DA=DMp-appstab}
	Let $B$ be a normal  variety over $K$. The functor $a_{tr}$ induces an equivalence of triangulated categories:
	\[
	{\LL a_{tr}}\colon{\RigDA_{\Frobet}}(B^{\Perf} )\cong{\RigDM_{\et}}(B^{\Perf} ).
	\]
\end{cor}

\begin{proof}
	Theorem \ref{DA=DM-app} states that the adjunction
	\[
	\adj{a_{tr}}{\Ch_{\Frobet,\B^1}\Psh(\RigSm/B^{\Perf})}{\Ch_{\Frobet,\B^1}\PST(\RigSm/B^{\Perf})}{o_{\tr}}
	\]
	is a Quillen equivalence. It therefore induces a Quillen equivalence on the categories of symmetric spectra
	\[
	\adj{a_{tr}}{\Sp^\Sigma_{T}\Ch_{\Frobet,\B^1}\Psh(\RigSm/B^{\Perf})}{\Sp^\Sigma_{T}\Ch_{\Frobet,\B^1}\PST(\RigSm/B^{\Perf})}{o_{\tr}}
	\]
	by means of \cite[Proposition 4.3.35]{ayoub-th2}.
\end{proof}

We now assume that $\Lambda$ equals $\Z$ if $\car K=0$ and equals $\Z[1/p]$ if $\car K=p$. 
In analogy with  the statement   $\DA_{\et}(B,\Lambda)\cong\DM_{\et}(B,\Lambda)$ proved for motives associated to schemes (see  \cite[Appendix B]{ayoub-etale}) it is expected that the following result also holds.

\begin{conj}
	Let $B$ be a  normal variety over $K$. The functors $(a_{tr}, o_{tr})$ induce an equivalence of triangulated categories:
	\[
	{\LL a_{tr}}\colon{\RigDA_{\et}}(B,\Lambda )\cong{\RigDM_{\et}}(B,\Lambda ).
	\]
\end{conj}

We remark that in the above statement differs from Corollary \ref{DA=DMp-appstab} for two main reasons: the ring of coefficients is   no longer assumed to be a $\Q$-algebra, and the class of maps with respect to which we localize are the $\et$-local maps and no longer the  $\Frobet$-local maps.

In order to reach this twofold generalization, using the techniques developed in 
\cite{ayoub-etale}, it would suffice to show the two following formal properties of the 2-functor $\RigDA_{\et}$:
\begin{itemize}
	\item \emph{Separateness}: for any $\Frob$-cover $B'\ra B$ the functor \[\RigDA_{\et}(B,\Lambda)\ra\RigDA_{\et}(B',\Lambda)\] is an equivalence of categories.
	\item \emph{Rigidity}: if $\car K\nmid N$ the functor \[\catD\Sh_{\et}(\Et/B,\Z/N\Z)\ra\RigDA_{\et}(B,\Z/N\Z)\] is an equivalence of categories,  where $\Et/B$ is the small \'etale site over $B$.
\end{itemize}

\section*{Acknowledgements}
This paper is part of my PhD thesis, 
carried out in a  co-tutelle program between the University of Milan and the University of Zurich. I am incredibly indebted to my advisor Professor Joseph Ayoub for having suggested me the main result of this paper, the outstanding amount of insight that he kindly shared with me, and his endless patience. I wish to express my gratitude to my co-advisor Professor Luca Barbieri Viale, for his invaluable encouragement and the numerous mathematical discussions throughout the development of my thesis. 
For the plentiful remarks   on preliminary versions of this work and his precious and friendly support, I also wish to thank Simon Pepin Lehalleur.

\bibliographystyle{plainnat}

 \end{document}